\documentclass[12pt]{article} 

\usepackage{amsmath, amsfonts, amssymb}
\usepackage{mathrsfs}
\usepackage{theorem}
\usepackage{bm}
\usepackage[dvipdf]{graphicx}

\RequirePackage[colorlinks,citecolor=blue,urlcolor=blue,bookmarksopen]{hyperref}

\newtheorem{mythm}{Theorem}[section]
\newtheorem{myprop}[mythm]{Proposition}
\newtheorem{mylem}[mythm]{Lemma}
\newtheorem{mycor}[mythm]{Corollary}

{}
{\newtheorem{myexam}{Example}[section]}

\def\R{\mathbb R}

\def\N{\mathbb N}
\def\C{\mathscr C}
\def\B{\mathscr B}

\def\F{\mathscr F}
\def\d{\text{\rm{d}}}
\def\E{\mathbb E}

\def\e{\text{\rm{e}}}

\def\veps{\varepsilon}

\def\S{\mathcal S}
\def\C{\mathscr C}

\def\pb{\mathscr{P}}

\def\wt{\widetilde}

\def\var{\mathrm{var}}
\def\W{\mathbb{W}}

\def\law{\mathcal{L}}
\def\ve{\varepsilon}

\newenvironment{proof}{{\noindent\it Proof.}\ }{\hfill $\square$\par}

\numberwithin{equation}{section}

\allowdisplaybreaks

\begin{document}
	
	\title{Averaging principle for two time-scale stochastic differential equations with fast component in noncompact space\footnote{Supported in part by National Key R\&D Program of China (No. 2022YFA1006000) and NNSFs of China (No. 12271397,12531007,12501190)}}
	
	\author{Shen Wang${}^a$    
\and Jinghai Shao${}^b$\thanks{a: College of Science, Civil Aviation University of China, Tianjin, China\\
\hspace*{1.52em} b: Center for Applied Mathematics, Tianjin University, Tianjin, China. Emails: shaojh@tju.edu.cn (Shao); wangs@cauc.edu.cn (Wang)}  
	 }
\date{}
\maketitle
	
\begin{abstract}
The asymptotic behavior for fully coupled multiscale stochastic systems becomes much complicated when the fast processes do not locate in a compact space. An example is constructed to show  that the averaged coefficients may become discontinuous even they are originally Lipschitz continuous when the fast process locate in a noncompact space. This work aims to reveal the impact of ergodicity of the fast process on the establishment of the averaging principle. The crucial point is to characterize  the continuous dependence of the invariant probability measure on parameters related to the slow process with respect to various distances in the Wasserstein space.
\end{abstract}



\section{Introduction}

In this work we are concerned with the following fully coupled two time-scale stochastic systems:
\begin{equation}\label{a-1}
\begin{cases}
  \d X_t^\veps=b(X_t^\veps,Y_t^\veps)\d t+\sigma(X_t^\veps,Y_t^\veps)\d W_t, &X_0^\veps=x_0,\\
  \d Y_t^\veps=\frac 1\veps f(X_t^\veps,Y_t^\veps)\d t+\frac1{\sqrt{\veps}} g(X_t^\ve, Y_t^\veps)\d B_t, &Y_0^\veps=y_0,
\end{cases}
\end{equation}
where $(W_t)$ and $(B_t)$ are $d$-dimensional mutually independent Wiener processes, $b(x,y)\in \R^d$ and $f(x,y)\in \R^d$ are drifts, $\sigma(x,y)\in \R^{d\times d}$ and $g(x,y)\in \R^{d\times d}$ are diffusion coefficients. The parameter $\veps>0$ represents the ratio between the time scale of processes $(X_t^\veps)$ and $(Y_t^\veps)$. We are interested in the case $\ve\ll 1$, in which $(X_t^\ve)$ is called the slow component, and $(Y_t^\ve)$ is called the fast component.

In the study of fully coupled stochastic systems, when the fast component does not locate in a compact space, the limit system becomes very complicated. Our purpose is to reveal the essential impact of various ergodicity of the fast component on the limit system. We shall show the wellposedness of the limit system and the averaging principle under different ergodicity of the fast component. To this end, it is crucial to show the continuous dependence of the invariant probability measure $\pi^x$ of the fast component $(Y_t^\veps)$ on the fixed state $x$ of the slow component $(X_t^\veps)$.

In applications many real systems can be viewed as a combination of slow and fast motions such as in modeling climate-weather interaction \cite{CMP,Ki01}, in biology \cite{KK}, in mathematical finance \cite{FFF,FFK} and references therein. The averaging principle has been established for various multiscale systems which says that the slow component $(X_t^\ve)$ will converge to some limit process $(\bar X_t)$ as $\ve\to 0$ in suitable sense. Also, there are many works devoted to the study of central limit theorems and large deviations of multiscale stochastic models. For a two-time scale system where both slow and fast components are continuous processes given as solutions of SDEs, these problems have been extensively studied  in e.g. \cite{Kh68,Kh05,Lip,Liu,Puh,Ver1,Ver2}, in \cite{HL20} for SDEs driven by fractional Brownian motions. For a two-time scale system where the slow component is a continuous process but the fast process is a jump process over a discrete space, these problems have been studied in e.g. \cite{BDG18,FL96,MS22} and references therein.

When the fast component does not depend on the slow component, we arrive at the classical uncoupled setup, and the averaging principle usually holds in quite general conditions(cf. e.g. \cite{PV1,PV3} and references therein). However, when the fast component depends on the slow component, the situation becomes more complicated. Anosov \cite{Ano} established the first relatively general result on fully coupled averaging principle. See the monograph of Kifer \cite{Ki09} for a detailed discussion on averaging of fully coupled dynamical systems. The averaging principle for fully coupled stochastic processes have been investigated in e.g. Freidlin and Wentzell \cite{FW} and Veretennikov \cite{Ve91}, the corresponding large deviation principle has been studied by Veretennikov \cite{Ver1,Ver2} and recently by Puhalskii \cite{Puh}. In \cite{PV2}, Pardoux and Veretennikov  analyzed the regularity of the density of invariant probability measure based on the PDE technique under certain smooth conditions of the coefficients. This work focuses on the averaging principle for fully coupled two time-scale stochastic processes.


To be precise, for each fixed slow component $x$, let $\pi^x$ denote the invariant probability measure of the fast component.
As usual, put $\bar{b}(x)=\int_{\R^d}b(x,y)\pi^x(\d y)$ and $\bar{\sigma}(x)$ the square root of $\bar a(x):=\int_{\R^d}(\sigma\sigma^\ast)(x,y)\pi^x(\d y)$, and consider the SDE
\begin{equation}\label{a-2}
\d \bar{X}_t=\bar{b}(\bar X_t)\d t+\bar{\sigma}(\bar{X}_t)\d W_t,\quad \bar{X}_0=x_0.
\end{equation}
 The averaging principle suggests that often
$(X_t^\ve)$ converges to a process $(\bar{X}_t)$.
To show this, a premise is the wellposedness of the SDE \eqref{a-2}.
In the uncoupled setup, Lipschitz continuity of $b$ and $\sigma$ implies already that $\bar{b}$ and $\bar{\sigma}$ are also Lipschitz continuous in $x$, and so there exists a unique solution $(\bar{X}_t)$ to SDE \eqref{a-2}. However, in the fully coupled case even when $b$ is Lipschitz continuous, we construct an example (see Example \ref{exam-2} in Section 2) to show that the averaged coefficients $\bar{b}(x)$ may even not be continuous in $x$, let alone Lipschitz. So, SDE \eqref{a-2} may not have solution at all.

The wellposedness of SDE \eqref{a-2} depends not only on the regularity of coefficients $b$ and $\sigma$ but also on the ergodicity of the fast component at each fixed slow component $x$ via the invariant probability measure $\pi^x$.
So, to establish the averaging principle we first study the continuity of $\pi^x$ in $x$ under  two kinds of strongly ergodic conditions for the fast component w.r.t.   total variation distance  and $L^1$-Wasserstein distance respectively.
As the space of probability measures is an infinite dimensional space, the previous mentioned distances are not equivalent, and so the continuity of $\pi^x$ w.r.t. different distances can be applied to deal with the regularity of $\bar{b},\,\bar{\sigma}$ for different kinds of $b$ and $\sigma$. Furthermore, under the wellposedness of the limit process $(\bar{X}_t)$ the averaging principle is established for   fully coupled two time-scale stochastic processes under certain conditions.

In the study of fully coupled diffusion processes, a widely used condition is the monotonicity condition on the fast component (cf. e.g. \cite{Givon,Liu}), that is,
there exists a constant $\beta>0$ such that
\begin{equation}\label{a-4}
  2(y_1\!-\!y_2)\cdot(f(x,y_1)\!-\!f(x,y_2))\!+\!\|g(x,y_1) \!-\!g(x,y_2)\|^2\!\leq \!-\beta |y_1\!-\!y_2|^2, \ x, y_1,y_2\in \! \R^d.
\end{equation}
Besides the monotonicity condition to characterize the ergodicity of the fast component, certain strict coercivity condition is also needed to establish the averaging principle (cf. \cite[Remark 2.1]{Liu}). 
Moreover,  Veretennikov \cite{Ve91} provided some very general conditions in the type of strong ergodicity theorem for the fast component. 
The intensive interaction between the fast and slow component of the two time-scale system makes this system much complicated. 
As explained by a constructed Example \ref{exam-2} below, we can see that $\bar b$ and $\bar{\sigma}$ may be discontinuous even $b$ and $\sigma$ are Lipschitz continuous. The reason to derive such phenomenon is due to the convergence rate of the fast component to its stable distribution, which is dependent on the slow component.

In view of our constructed example, we shall study the averaging principle for the two time-scale systems in terms of the ergodicity condition of the semigroup of the fast components. We shall impose strong ergodicity condition on the semigroup of the fast component using respectively the total variation distance and the $L^1$-Wasserstein distance to measure the convergence of the semigroup to its invariant probability measure. The strong ergodicity of diffusion processes is well studied topic, and there are many criteria in the existing literature. The monotonicity condition is a sufficient condition to guarantee the strong ergodicity in the $L^1$-Wasserstein distance. Accordingly, the continuity of invariant probability  measure $\pi^x$ in $x$ is proved, which helps us to show the continuity of $\bar{b}$ and $\bar{\sigma}$ and further the wellposedness of the SDE for limit process $(\bar{X}_t)$. At last, we can use the time discretization method and the coupling method to establish the averaging principle.

This work is organized as follows. In Section \ref{sec-2}, we first   construct an example in Subsection \ref{subsec-2.1} to show the essential impact of $\pi^x$ on the regularity of $\bar{b}(x)$, $\bar{\sigma}(x)$. Then, we investigate the continuity of $\pi^x$ in $x$ w.r.t. different distances in Subsection \ref{subsec-2.2}. In Section \ref{sec-3}, the averaging principle for the fully coupled system \eqref{a-1} is established under different  conditions.

\section{Continuous dependence on parameters of invariant probability measures} \label{sec-2}

\subsection{An illustrative example}\label{subsec-2.1}
In this part we aim to present the complexity of fully coupled stochastic systems via  an explicit example.

\begin{myexam}\label{exam-2}
Let $(X_t^\ve)$ and $(Y_t^\ve)$ be stochastic processes respectively on $[0,1]$ and on $[0,\infty)$ with reflection boundary satisfying
\begin{equation}\label{ex-2}
\begin{cases}
\d X_t^\ve= Y_t^\ve \d t+ Y_t^\ve \d W_t, & X_0^\ve=x_0\in (0,1),\\
\d Y_t^\ve= \frac1{\ve} \tilde f(X_t^\ve, Y_t^\ve)\d t+\frac1{\sqrt{\ve}} \d B_t, &Y_0^\ve=y_0\in (0,\infty),
\end{cases}
\end{equation}
where
\[\tilde f(x,y)=\frac{-x^3 \e^{-xy}-(1-x)\e^{-y}}{x^2\e^{-xy}+(1-x)\e^{-y}}, \quad x\in [0,1], \,y\in [0,\infty).\]
\end{myexam}
As $\tilde f(x,y) $ is continuous on $[0,1]\times[0,\infty)$, the solution to \eqref{ex-2} exists and is unique in distribution according to \cite{SV79}. In this example, $b(x,y)=\sigma(x,y)=y$ and $g(x,y)=1$ are both Lipschitz continuous.
%

For each $x\in [0,1]$, the invariant probability  measure $\pi^x$ associated with the SDE
\begin{equation}\label{ex-3}
\d Y_t^{x,y}=\tilde f(x,Y_t^{x,y})\d t+\d B_t,\quad Y_0^{x,y}=y\in (0,\infty),
\end{equation} is given by
\[\pi^x(\d y)=\big(x^2\e^{-xy}+(1-x)\e^{-y}\big)\d y.\]
Then
\begin{align*}
  \bar{b}(x)&:=\int_0^\infty \!\!b(x,y)\pi^x(\d y)=\int_0^\infty\!\! y \big(x^2\e^{-xy}\!+\! (1\!-\!x) \e^{-y})\d y
  =\begin{cases} 2\!-\!x, &\text{if}\ x\!\in\! (0,1],\\ 1, &\text{if}\ x=0.\end{cases}
\end{align*}
It is clear that $\bar{b}(x)$ is not continuous at $x=0$. Hence, this example is our desired example to show the complexity of the limit behavior of the fully coupled system $(X_t^\ve, Y_t^\ve)$ as $\ve\to 0$.

Next, we investigate the ergodic property of the process $(Y_t^{x,y})$ given by \eqref{ex-3}. To this end, introduce the notation
\begin{gather*}
C_x(y)
=\ln\big(x^2 \e^{-xy}+(1-x)\e^{-y}\big)
 , \ \ 
\mu^x(\d y)=\e^{C_x(y)}\d y.
\end{gather*} So, $\mu^x([0,\infty))=1$ and  $\pi^x(\d y)=\mu^x(\d y)/\mu^x([0,\infty))=\mu^x(\d y)$.

1) The process $(Y_t^{x,y})$ is unique and ergodic. To this end, one needs to check $\int_0^\infty \mu^x([0,y])\e^{-C_x(y)}\d y =\infty$ and $\mu^x([0,\infty))<\infty$.
Indeed, $\mu^x([0,\infty))=1$,
\begin{align*}
&\int_0^\infty \Big(\int_0^y \e^{C_x(z)}\d z\Big) \e^{-C_x(y)}\d y\\
&=\int_0^\infty \Big(\int_0^y \big(x^2\e^{-xz}+(1-x)\e^{-z}\big)\d z\Big) \frac{\d y}{ x^2\e^{-xy}+(1-x)\e^{-y}} \\
&=\int_0^\infty \frac{x(1-\e^{-xy})+(1-x)(1-\e^{-y})}{x^2\e^{-xy}+(1-x)\e^{-y}}\d y\\
&\geq \int_1^\infty \frac{x(1-\e^{-x})+(1-x)(1-\e^{-1})}{x^2\e^{-x}+(1-x)\e^{-1}}\d y=\infty.
\end{align*}
This implies that for each $x\in [0,1]$ the process $(Y_t^{x,y})$ is ergodic on $[0,\infty)$ by the criterion of ergodicity for one-dimensional diffusion processes; see, for instance,  \cite[Chapter 5]{Ch05}.

2) According to \cite[Chapter 5]{Ch05}, to study the exponential ergodicity of $(Y_t^{x,y})$, we need to study
\begin{align*}
  &\mu^x([z,\infty))\int_0^z\e^{-C_x(y)}\d y
   =\big(x\e^{-xz}+(1-x)\e^{-z}\big)\int_0^z\frac{\d y}{x^2\e^{-xy}+(1-x)\e^{-y}}.
\end{align*}
Using L'H\^opital's rule,
\begin{align*}
  &\lim_{z\to \infty}  \big(x\e^{-xz}\!+\! (1\!-\! x)\e^{-z}\big)\int_0^z\frac{\d y}{x^2\e^{-xy}\!+\! (1\!-\!x)\e^{-y}}  =\lim_{z\to \infty}\Big(\frac{x\e^{(1\! -\! x)z}+1\!-\! x}{x^2\e^{(1-x)z}+1\!-\!x}\Big)^2<\infty.
\end{align*}
Hence, we have  $\sup_{z>0}\big\{\mu^x([z,\infty))\int_0^z\e^{-C_x(y)}\d y\big\} <\infty$, and then the process $(Y_t^{x,y})$ is exponentially ergodic.

3) According to \cite[Theorem 2.1]{Mao}, $(Y_t^{x,y})$ is strongly ergodic (also called uniformly ergodic) if and only if \[\int_0^\infty \mu^x([y,\infty))\e^{-C_x(y)}\d y <\infty.\]
However, for this example,
\begin{align*}
   &\int_0^\infty \mu^x([y,\infty))\e^{-C_x(y)}\d y=\int_0^\infty \frac{ x\e^{-xy}+(1-x)\e^{-y}}{ x^2\e^{-xy}+(1-x)\e^{-y}}\d y=\infty,
\end{align*}
hence, $(Y_t^{x,y})$ is not strongly ergodic.

\subsection{Continuity of $ \pi^x$    in the Wasserstein space}\label{subsec-2.2}

The space of probability measures over $\R^d$, denoted by $\pb(\R^d)$, is an infinite dimensional space, on which various distances have been defined. These distances are not mutually equivalent. To deal with different coefficients, one needs to consider the continuity of $\pi^x$ in $x$ in different distances.  We shall consider three kinds of distances on $\pb(\R^d)$ including the total variation distance $\|\cdot\|_\var$, $L_1$-Wasserstein distance $\W_1$, bounded Lipschitz distance $\W_{bL}$ (also called Fortet-Mourier distance). These distances are defined as follows: for two probability measures $\mu,\nu$ on $\R^d$,
\begin{align*}
  \|\mu-\nu\|_\var&:=2\sup_{A\in\B(\R^d)}|\mu(A)-\nu(A)|=\sup_{|h|\leq 1} |\mu(h)-\nu(h)|,\\
  \W_p(\mu,\nu)&:=\inf\Big\{\int_{\R^d\times\R^d}\!\!|x-y|^p\,\Gamma(\d x,\d y);\Gamma\in \C(\mu,\nu)\Big\}^{\frac 1p},\quad p\geq 1,\\
  \W_{bL}(\mu,\nu)&:=\sup\Big\{\mu(h)-\nu(h); |h(x)|\leq 1, \ |h(x)-h(y)|\leq |x-y|, x,y\in \R^d\Big\}.
\end{align*}
By Kantorovich's dual representation theorem for the Wasserstein distance,
\[\W_1(\mu,\nu)=\sup\Big\{\mu(h)-\nu(h); |h(x)-h(y)|\leq |x-y|, x,y\in \R^d\Big\}.\]
We refer the readers to \cite[Chapter 6]{Vill} for more discussion on the relationship of various distances on $\pb(\R^d)$.

Associated with the two time-scale system $(X_t^\veps, Y_t^\veps)$, we introduce the following  process $(Y_t^{x,y})$: for each given $x,y\in \R^d$:
\begin{equation}\label{a-3}
\d Y_t^{x,y}= f(x, Y_t^{x,y})\d t+g(x, Y_t^{x,y})\d B_t,\quad Y_0^{x,y}=y.
\end{equation}
In this work we shall use two kinds of ergodicity condition on $(Y_t^{x,y})$ with respect to the total variation distance and the $L_1$-Wasserstein distance respectively. Throughout this section, $P_t^{x}(y,\cdot)$ denotes the semigroup associated with the Markov process $(Y_t^{x,y})$ given by \eqref{a-3}, i.e. $P_t^x h(y):=\E[h(Y_t^{x,y})]$ for $h\in \B_b(\R^d)$.  We assume in this work that there is a unique invariant probability measure $\pi^x$ to the semigroup $P_t^x$. 
\begin{itemize}
\item[$\mathrm{(E1)}$] There exist  constants $\kappa_1,\lambda_1>0$ such that
\[\sup_{y\in \R^d} \|P_t^x(y,\cdot)-\pi^x\|_\var\leq \kappa_1\e^{-\lambda_1 t},\quad t>0, x\in \R^d.\]
\item[$\mathrm{(E2)}$] There exist constants $\kappa_2,\lambda_2>0$ such that
\[\sup_{y\in \R^d} \W_1(P_t^x(y,\cdot),\pi^x)\leq \kappa_2\e^{-\lambda_2 t},\quad t>0, x\in \R^d.\]
\end{itemize}
Ergodicity for diffusion processes is an extensively studied topic (cf. \cite{Chen-1,MT}). We shall present some criteria on the coefficients $f,\,g$ to verify (E1) and (E2) in Section 3; see, Lemma \ref{lem-3.1} and Lemma \ref{lem-3.2}.

Now we collect  the conditions on the coefficients of $(X_t^\ve,Y_t^\ve)$ used in this work.

\noindent\textbf{Assumptions on slow component}:
\begin{itemize}
  \item[$\mathrm{(A1)}$] There exists  $K_1>0$ such that for any $x_1,x_2,y_1,y_2\in \R^d$,
  \[|b(x_1,y_1)-b(x_2,y_2)|^2+\|\sigma(x_1,y_1)-\sigma(x_2,y_2)\|^2\leq K_1\big(|x_1-x_2|^2+|y_1-y_2|^2\big) .\]
  \item[$\mathrm{(A2)}$] There exists  $K_2>0$ such that
  $\displaystyle |b(x,y)|+\|\sigma(x,y)\|\leq K_2$ for any $ x,y\in \R^d$.
  \item[$\mathrm{(A3)}$] Assume that
  $\displaystyle \inf_{x,y\in\R^d}\inf_{\xi\in\R^d,|\xi|=1}\xi^\ast a(x,y)\xi>0$,  where $a(x,y)=(\sigma\sigma^\ast)(x,y)$.
\end{itemize}
\noindent\textbf{Assumptions on fast component}:
\begin{itemize}


  \item[$\mathrm{(B1)}$] There exists $K_3>0$ such that for $x_1,x_2,y,z,y_1,y_2\in \R^d$,
  \begin{gather*}
  (f(x_1,y)-f(x_2,y))\cdot z\leq K_3|x_1-x_2||z|, \\  
  (f(x_1,y_1)\!-\!f(x_2,y_2))\cdot(y_1\!-\!y_2)\!+\!\|g(x_1,y_1)\!-\!g(x_2,y_2) \|^2\!\leq \!K_3(|x_1\!-\!x_2|^2\!+\!|y_1\!-\!y_2|^2).
  \end{gather*} 


  \item[$\mathrm{(B2)}$] There exists   $\lambda_3>0$ such that   \[\eta^\ast (g g^\ast)(x,y)\eta\geq \lambda_3 |\eta|^2,\ \  \text{$ x, y,\eta \in \R^d$}.\]
  \item[$\mathrm{(B3)}$] There is  a nonnegative function $ \vartheta\in C(0,\infty)$ such that $\int_0^\infty \vartheta(r)\d r<\infty$  and 
      \begin{equation}\label{a-6}
      \sup_{|y_1-y_2|=r}\frac 1r\big\{ \|\tilde g(x,y_1)-\tilde g(x,y_2)\|^2+2(f(x,y_1)-f(x,y_2))\cdot (y_1-y_2)\big\}\leq \vartheta (r)  
      \end{equation}for all $r>0$, $ x,y_1,y_2\in \R^d$, 
      where $\tilde g(x,y)=\sqrt{(gg^\ast)(x,y)-\lambda_3\mathrm{I}}$ and $\mathrm{I}$ denotes the $d\times d$ identity matrix.  
\end{itemize}

\begin{myprop}\label{prop-2.1}
Assume $\mathrm{(E1)}$ and  $\mathrm{(B1)}$-$\mathrm{(B3)}$ hold. In addition, assume $g(x,y)$ depends only on $y$, i.e. $g(x,y)=g(y)$. Then there exists a constant $C$ such that
\begin{equation}\label{b}\|\pi^{x_1}-\pi^{x_2}\|_{\var}\leq C |x_1-x_2|^{2/3},\quad x_1,x_2\in\R^d.
\end{equation}
\end{myprop}

\begin{proof}
  For any $h\in C(\R^d)$ with $|h|_\infty:=\sup_{x\in\R^d}|h(x)|\leq 1$, due to (E1),
  \begin{equation}\label{b-1}
  \begin{split}
    |\pi^{x_1}(h)-\pi^{x_2}(h)|
    &\leq \big| \pi^{x_1}(h)-\frac 1t\int_0^t\!P_s^{x_1}h(y_0)\d s\big|+\big|\pi^{x_2}(h)-\frac 1t\int_0^t\!P_s^{x_2}h(y_0)\d s\big|\\
    &\quad+\big|\frac1t\int_0^t\!|P_s^{x_1}h(y_0)-P_s^{x_2}h(y_0)|\d s\big|\\
    &\leq \frac 1t\int_0^t\!\big(\|P_s^{x_1} (y_0,\cdot)-\pi^{x_1}\|_\var\! +\! \|P_s^{x_2}(y_0,\cdot) -\pi^{x_2}\|_\var\big)\d s\\
    &\quad+\frac 1t\int_0^t\!|P_s^{x_1}h(y_0)-P_s^{x_2}h(y_0)|\d s\\
    &\leq \frac {2\kappa_1}t\int_0^t\!\e^{-\lambda_1 s}\d s+\frac 1t\int_0^t\!|P_s^{x_1}h(y_0)-P_s^{x_2}h(y_0)|\d s.
  \end{split}
  \end{equation}
  Consider the function $\phi(r):=P^{x_1}_{r}(P_{s-r}^{x_2}h) (y_0)$ for $r\in [0,s]$. It is easy to see
  \begin{equation}\label{b-1.5}
  P_s^{x_1}h(y_0)-P_s^{x_2}h(y_0)=\int_0^s\phi'(r)\d r=\int_0^s\! P_r^{x_1}(\mathscr{L}^{x_1}-\mathscr{L}^{x_2}) P_{s-r}^{x_2}h(y_0)\d r.
  \end{equation}
According to \cite[Theorem 3.4]{PW06}, it follows from (B1)-(B3) and $\Theta:=\int_0^\infty \vartheta(r)\d r<\infty$ that  
\begin{equation}\label{G-2}
|\nabla P_t^x h|_\infty\leq \frac{\e^{\frac{\Theta}{4\lambda_3}}(1+2\lambda_3)} {2\lambda_3\sqrt{t}} |h|_\infty,\quad t>0.
\end{equation}
Then, by (B1) and  \eqref{G-2}, it follows from \eqref{b-1.5} that for $s\in [0,t]$
  \begin{align*}
    |P_s^{x_1} h(y_0)-P_s^{x_2}h(y_0)|&\leq \int_0^s\!\sup_{y\in \R^d}|(f(x_1,y)-f(x_2,y))\cdot\nabla P_{s-r}^{x_2}h(y)| \d r\\
    &\leq K_3|x_1-x_2|\int_0^s\! \frac{\e^{\frac{\Theta}{4\lambda_3}}(1+2\lambda_3)} {2\lambda_3\sqrt{r}}\d r\\
    &\leq C \sqrt{s}|x_1-x_2|.
  \end{align*}
Inserting this estimate into \eqref{b-1}, we get
\begin{align*}
  |\pi^{x_1}(h)-\pi^{x_2}(h)|&\leq \frac {2\kappa_1}{t}(1-\e^{-\lambda_1 t})+C \sqrt{t}|x_1-x_2|.
\end{align*} Furthermore,  by taking  $t=|x_1-x_2|^{-2/3}$, it holds for some $C>0$ that
\begin{equation}\label{b-2}
|\pi^{x_1}(h)-\pi^{x_2}(h)|\leq C|x_1-x_2|^{2/3}.
\end{equation}
By the arbitrariness of $h\in C(\R^d)$ with $|h|_\infty \leq 1$, we get that
\begin{align*}
\|\pi^{x_1}-\pi^{x_2}\|_{\mathbb{V}}&:=
\sup\{|\pi^{x_1}(h)-\pi^{x_2}(h)|;\ h\in C(\R^d), |h|_\infty\leq 1\}\leq C|x_1-x_2|^{2/3}.
\end{align*}

To show the desired conclusion \eqref{b}, we only need to show
\begin{equation}\label{b-2.5}
\|\pi^{x_1}-\pi^{x_2}\|_{\mathbb{V}}=\|\pi^{x_1}-\pi^{x_2}\|_{\var}.
\end{equation}
Indeed, by Lusin's theorem (cf. \cite[Theorem 2.23]{Rud}), for every $h\in \B(\R^d)$ with $|h|_\infty\leq 1$, for any $\veps>0$, there exists a function $h_\veps\in C_c(\R^d)$ such that
\[(\pi^{x_1}+\pi^{x_2})\big(\{h\neq h_\veps\}\big)<\veps,\quad \text{and}\ |h_\veps|_\infty\leq |h|_\infty.\]
Hence,
\begin{align*}
  |\pi^{x_1}(h)\!-\!\pi^{x_2}(h)|&\leq |\pi^{x_1}(h_\veps)\!-\!\pi^{x_2}(h_\ve)|\!+\! |\pi^{x_1}(h-h_\ve)|\!+\!|\pi^{x_2}(h-h_\ve)|\\
  &\leq |\pi^{x_1}(h_\ve)-\pi^{x_2}(h_\ve)|+2\veps \leq \|\pi^{x_1}-\pi^{x_2}\|_{\mathbb{V}}+2\veps.
\end{align*}
From this, it is easy to see \eqref{b-2.5} holds. Thus, the proof is completed.
\end{proof}

\begin{mycor} \label{cor-1}
Let the conditions of Proposition \ref{prop-2.1} be valid. In addition,
  suppose $\mathrm{ (A1)}$ and $\mathrm{(A2)}$ hold. Let
\begin{equation}\label{b-0}
\bar{b}(x)=\int_{\R^d}b(x,y)\pi^x(\d y),\ \bar{a}(x)=\int_{\R^d} a(x,y)\pi^x(\d y).
\end{equation}
Then $\bar{b}$ and $\bar{a}$ are  locally H\"older continuous of exponent $2/3$.
\end{mycor}

\begin{proof}
By \eqref{b},  (A1), and (A2), we obtain that
\begin{align*}
  |\bar{b}(x_1)-\bar{b}(x_2)|&\leq \Big|\int_{\R^d}\!b(x_1,y)\big(\pi^{x_1}(\d y)- \pi^{x_2}(\d y)\big)\Big|\!+\!\int_{\R^d}\!|b(x_1,y)-b(x_2,y)|\pi^{x_2}(\d y)\\
  &\leq \sup_{y\in\R^d} |b(x_1,y)|\|\pi^{x_1}-\pi^{x_2}\|_\var+\sqrt{K_1}|x_1-x_2|\\
  &\leq   K_2(1+|x_1|) |x_1-x_2|^{2/3}+\sqrt{K_1}|x_1-x_2|,
\end{align*} which implies that $\bar{b}$ is locally H\"older continuous of exponent $2/3$.
Similar deduction yields the result for $\bar{a}$. Hence, this corollary is proved.
\end{proof}

\begin{myexam}\label{exam-3}
Let us provide a simple example to ensure all the conditions of Proposition \ref{prop-2.1} hold. Consider the following SDE on the line:
\[\d Y_t= (X_t-Y_t^3) \d t+\d B_t,\quad Y_0=y_0\in \R.\]
Then, we have $f(x,y)=x-y^3$ and $g(x,y)=1$. It is easy to see (B1)-(B3) hold with $\vartheta(r)=\frac1{1+r^2}$ for $r>0$. Moreover, according to Lemma \ref{lem-3.1} below, one can check directly that (E1) holds. Hence, all the conditions of Proposition \ref{prop-2.1} hold. 
\end{myexam}

\begin{myprop}\label{prop-2.2}
$\mathrm{(i)}$ Assume $\mathrm{(E1),\, (B1) }$ hold. Then for $x_1,x_2\in \R^d$
\begin{equation}\label{c-1}\W_{bL}(\pi^{x_1},\pi^{x_2})\leq |x_1-x_2|\mathbf{1}_{\{|x_1-x_2|\geq 2\kappa_1\}}+|x_1-x_2|^{\frac{\lambda_1}{\lambda_1+K_3}} \mathbf{1}_{\{|x_1-x_2|<2\kappa_1\}}.
\end{equation}

$\mathrm{(ii)}$ Assume $\mathrm{(E2),\, (B1)}$ hold. Then for $x_1,x_2\in \R^d$
\begin{equation}\label{c-2}
\W_1(\pi^{x_1},\pi^{x_2})\leq |x_1-x_2|\mathbf{1}_{\{|x_1-x_2|\geq 2\kappa_2\}}+|x_1-x_2|^{\frac{\lambda_2}{\lambda_2+K_3}} \mathbf{1}_{\{|x_1-x_2|<2\kappa_2\}}.
\end{equation}
\end{myprop}

\begin{proof}
  (i) For $x_1,x_2,y\in \R^d$, consider
  \begin{align*}
    \d Y_t^{x_1,y}&=f(x_1,Y_t^{x_1,y})\d t+g(x_1,Y_t^{x_1,y})\d B_t,\quad Y_0^{x_1,y}=y,\\
    \d Y_t^{x_2,y}&=f(x_2,Y_t^{x_2,y})\d t+g(x_2,Y_t^{x_2,y})\d B_t,\quad Y_0^{x_2,y}=y.
  \end{align*}
  Then, by It\^o's formula and (B1),
  \[\E|Y_t^{x_1,y}-Y_t^{x_2,y}|^2\leq \E\int_0^t\!K_3\big(|x_1-x_2|^2+|Y_s^{x_1,y}-Y_s^{x_2,y}|^2\big)\d s.
  \]
  This yields
  \begin{equation}\label{c-3}
  \E|Y_t^{x_1,y}-Y_t^{x_2,y}|^2\leq \big(\e^{K_3t}-1\big)|x_1-x_2|^2.
  \end{equation}

  For any $h\in \B(\R^d)$ with $|h|_\infty \leq 1$ and $|h|_{\mathrm{Lip}}:=\sup_{z_1\neq z_2} \frac{|h(z_1)-h(z_2)|}{|z_1-z_2|}\leq 1$, by virtue of \eqref{c-3} and (E1),
  \begin{equation}\label{c-4}
  \begin{split}
    &|\pi^{x_1}(h)-\pi^{x_2}(h)|\\
    &\leq |\pi^{x_1}(h)\!-\! P_t^{x_1}h(y)|\! +\! |\pi^{x_2}(h)\! -\! P_t^{x_2} h(y)| 
     \!+\!|P_t^{x_1}h(y)\! -\! P_t^{x_2}h(y)|\\
    &\leq \|P_t^{x_1}(y,\cdot)\! -\! \pi^{x_1}\|_{\var}\!+ \! \|P_t^{x_2}(y,\cdot)\! -\! \pi^{x_2}\|_{\var} \! +\! \E|h(Y_t^{x_1,y})-h(Y_t^{x_2,y})|\\
    &\leq 2\kappa_1\e^{-\lambda_1 t}+|h|_{\mathrm{Lip}}\big(\E|Y_t^{x_1,y}-Y_t^{x_2,y}|^2\big)^{1/2}\\
    &\leq 2\kappa_1 \e^{-\lambda_1 t}+\big(\e^{K_3 t}-1\big)^{1/2} |x_1-x_2|=:\mathrm{(I)}.
  \end{split}
  \end{equation}
  To estimate $\mathrm{(I)}$, when $|x_1-x_2|\geq 2\kappa_1$, it holds
  \[\mathrm{(I)}\leq |x_1-x_2|+\big(\e^{K_3 t}-1\big)^{1/2} |x_1-x_2|,\]
  which yields $\mathrm{(I)}\leq |x_1-x_2|$ by letting $t\to 0$.
  When $|x_1-x_2|<2\kappa_1$, take $t=-\frac 1{\lambda_1} \ln \frac{|x_1-x_2|^p}{2\kappa_1}$ for some $p>0$ to be determined later. Then $\e^{-\lambda_1 t}=\frac{|x_1-x_2|^p}{2\kappa_1}$, and
  \begin{align*}
    \mathrm{(I)}&\leq |x_1\! -\! x_2|^p\!+\! \e^{K_3 t}|x_1\!-\! x_2|
    \leq |x_1\!-\!x_2|^p+(2\kappa_1)^{\frac{K_3 p}{\lambda_1}} |x_1\!-\!x_2|^{1-\frac{K_3 p}{\lambda_1}}.
  \end{align*}
  Take $p=\frac{\lambda_1}{\lambda_1+K_3}$, then $p=1-\frac{K_3 p}{\lambda_1}$.
  Finally, $\mathrm{(I)}\leq \big(1+(2\kappa_1)^{\frac{K_3}{\lambda_1+K_3}} \big) |x_1-x_2|^{\frac{\lambda_1}{\lambda_1+K_3}}$.
  By the arbitrariness of $h$ with $|h|_\infty \leq 1$ and $|h|_{\mathrm{Lip}}\leq 1$, we get \eqref{c-1} from the definition of $\W_{bL}(\pi^{x_1},\pi^{x_2})$.

  (ii) For any $h\in C(\R^d)$ with $|h|_{\mathrm{Lip}}\leq 1$, for $y\in \R^d$,  due to \eqref{c-3},
  \begin{align*}
    &|\pi^{x_1}(h)-\pi^{x_2}(h)|\\
    &\leq |\pi^{x_1}(h)-P_t^{x_1}h(y)|+|\pi^{x_2}(h)-P_t^{x_2}h(y)| + |P_t^{x_1} h(y)-P_t^{x_2}h(y)|\\
    &\leq \W_1(P_t^{x_1}(y,\cdot), \pi^{x_1})+\W_1(P_t^{x_2}(y,\cdot), \pi^{x_2}) +|h|_{\mathrm{Lip}}\big(\E|Y_t^{x_1,y}- Y_t^{x_2,y}|^2\big)^{\frac 12}\\
    &\leq 2\kappa_2\e^{-\lambda_2 t} +\big(\e^{K_3 t}-1\big)^{\frac 12} |x_1-x_2|.
  \end{align*} Then, completely similar to assertion (i), by choosing suitable $t>0$, we have
  \[|\pi^{x_1}(h)-\pi^{x_2}(h)|\leq |x_1-x_2|\mathbf{1}_{\{|x_1-x_2|\geq 2\kappa_2\}}+|x_1-x_2|^{\frac{\lambda_2}{\lambda_2+K_3}} \mathbf{1}_{\{|x_1-x_2|<2\kappa_2\}}, \]
  which yields \eqref{c-2} by taking supremum over $h\in C(\R^d)$ with $|h|_{\mathrm{Lip}}\leq 1$.
\end{proof}

\begin{mycor}\label{cor-2}
$\mathrm{(i)}$ Assume that $\mathrm{ (E1),\, (B1)}$ and $\mathrm{(A1),\,(A2)}$ hold. Then $\bar{b}$, $\bar{a}$ defined in \eqref{b-0} are locally H\"older continuous of exponent $\frac{\lambda_1}{\lambda_1+K_3}$.

$\mathrm{(ii)}$ Assume $\mathrm{(E2),\, (B1)}$ and $\mathrm{(A1)}$ hold. Then $\bar{b}$, $\bar{a}$ are H\"older continuous of exponent $\frac{\lambda_2}{\lambda_2+K_3}$.
\end{mycor}

\begin{proof} (i) We only prove the assertion for $\bar{b}$, and the assertion for $\bar{a}$ can be proved in the same way.
  By (A1) and (A2), for $x_1,x_2\in \R^d$,
  \begin{align*}
    |\bar{b}(x_1)-\bar{b}(x_2)|&\leq \big|\int_{\R^d}\!b(x_1,y)(\pi^{x_1} (\d y)-\pi^{x_2}(\d y)\big| +\int_{\R^d}\! |b(x_1,y)-b(x_2,y)|\pi^{x_2}(\d y)\\
    &\leq K_2(1+|x_1|) \|\pi^{x_1}-\pi^{x_2}\|_{\var} +\sqrt{K_1}|x_1-x_2|.
  \end{align*}
  Invoking \eqref{c-1}, it is clear that for any $R>0$
  \[\sup_{\substack{|x_1|,|x_2|\leq R,\\ x_1\neq x_2}}\!\frac{\ |\bar{b}(x_1)-\bar{b}(x_2)|\ }{|x_1-x_2|^{\frac{\lambda_1} {\lambda_1 +K_3}}} <\infty.\]
  Thus, $\bar{b}$ is locally H\"older continuous of exponent $\frac{\lambda_1}{\lambda_1+K_3}$.

  (ii) By (A1),  it holds $|b(x_1,y)-b(x_1,y')|\leq \sqrt{K_1}|y-y'|$ for $y,y'\in \R^d$. Using \eqref{c-2},
  \begin{align*}
    |\bar{b}(x_1)-\bar{b}(x_2)|&\leq \! \big|\!\int_{\R^d}\!\!b(x_1,y)\big(\pi^{x_1}(\d y)\!-\!\pi^{x_2}(\d y)\big)\big|\!+\!
     \int_{\R^d}\!|b(x_1,y)\!-\!b(x_2,y)|\pi^{x_2}(\d y)\\
     &\leq \sqrt{K_1} \W_1(\pi^{x_1},\pi^{x_2})+\sqrt{K_1}|x_1-x_2|\\
     &\leq \sqrt{K_1}\big(|x_1-x_2|^{\frac{\lambda_2}{\lambda_2+K_3}} +2|x_1-x_2|\big).
  \end{align*}
  This yields immediately that $\bar{b}$ is H\"older continuous of exponent $\frac{\lambda_2}{\lambda_2+K_3}$.   
\end{proof}

\section{Averaging principle for two time-scale systems}\label{sec-3}

Before establishing the averaging principle using the ergodicity condition (E1) or (E2), we would like introduce some sufficient conditions to verify (E1) and (E2) based on the coefficients $f,\,g$ of the fast component. We refer the readers to the monograph \cite{Ch05} for more discussion on (E1), especially its connection of functional inequalities, that is,  Nash inequality can yield (E1) (cf. \cite[Chapter 1]{Ch05}).  Recall that the diffusion process $(Y_t^{x,y})$ is defined in \eqref{a-3} for each fixed slow component $x$ and its semigroup is denoted by $P_t^{x}(y,\cdot)$.

\begin{mylem}[\cite{Mao}]\label{lem-3.1}
(1) When $(Y_t^{x,y})$ is a diffusion process on $[0,\infty)$ with reflection boundary at $0$. Assume $g(x,y)>0$ for $x,y\in[0,\infty)$. Let $C_x(y)=\int_1^y \frac{f(x,z)}{g^2(x,z)}\d z$. Then $(Y_t^{x,y})$ is strongly ergodic in the sense of (E1) if and only if
\begin{gather*}
  \int_0^\infty\!\!\e^{-C_x(y)} \Big(\int_0^y\!\! g^{-2}(x,z) \e^{C_x(z)}\d z\Big)\d y =\infty, \quad 
  \int_0^\infty \!\!\e^{-C_x(y)}\Big(\int_y^\infty\!\!  g^{-2}(x,z)\e^{C_x(z)}\d z\Big)\d y<\infty.
\end{gather*}
(2) Generally, assume $(Y_t^{x,y})$ is a reversible process in $\R^d$, i.e. for each $x\in \R^d$ there exists a function $V^x\in C^2(\R^d)$ satisfying $\int_{\R^d}\e^{V^x(y)}\d y<\infty$,
\[g(x,y)=\sum_{j=1}^d G_{ij}(x,y)\frac{\partial}{\partial y_j}V^x(y)+\sum_{j=1}^d \frac{\partial}{\partial y_j} G_{ij}(x,y),
\] where $G(x,y)=\big(G_{ij}(x,y)\big)=(g g^\ast)(x,y)$. Suppose that
$G(x,y)=\mathrm{diag}(G_{ii}(x,y))$ and $G_{ii}(x,y)=G_{ii}(x,y_i)$, $i=1,\ldots, d$, for $x\in \R^d$, $y=(y_1,\ldots,y_d)\in \R^d$.
Let
\begin{align*}
\rho_x(y,y')&=\Big(\sum_{i=1}^d\Big(\int_{y_i}^{y_i'} G_{ii}(x,z)^{-\frac 12}\d z\Big)^2\Big)^{\frac 12},\\
h_i^x&=\sqrt{G_{ii}(x,y)} \frac{\partial}{\partial y_i}V^x(y)+
\big(2\sqrt{G_{ii}(x,y)}\big)^{-1} \frac{\partial}{\partial y_i} G_{ii}(x,y),\\
\gamma_x(r)&=\sup_{\rho_x(y,y')=r}\sum_{i=1}^d \big(h_i^x(y)-h_i^x( y')\big)\int_{y_i}^{y_i'}\!G_{ii}(x,z)^{-\frac 12}\d z,\\
C_x(r)&=\exp\Big[\int_0^r \frac{\gamma_x(s)}{4s}\d s\Big].
\end{align*}
If
$\int_0^\infty\!\Big(C_x(s)^{-1}\int_s^\infty \frac{C_x(r)}{4}\d r\Big)\d s <\infty,$
then $(Y_t^{x,y})$ is strongly ergodic in the sense of $\mathrm{(E1)}$.
\end{mylem}

The exponential decay of the semigroup $(P_t^x)$ in the Wasserstein distance is heavily related to geometry of the underlying space as illustrated in \cite{RS}. The best convergence rate in the inequality is called the Wasserstein curvature in \cite{Jou} or the coarse Ricci curvature in \cite{Oll}. Here, we shall provide a criterion for (E2) based on the coupling method.

Recall the generator $\mathscr{L}^x$ of $(Y_t^{x,y})$ given by
\[\mathscr{L}^x=\frac 12\sum_{k,l=1}^d G_{kl}(x,y) \frac{\partial^2}{\partial y_k\partial y_l}+\sum_{k=1}^df_k(x,y)\frac{\partial}{\partial y_k}. \]
For simplicity, we write $\mathscr{L}^x\sim (G(x,y),f(x,y))$. An operator $\widetilde{\mathscr{L}}$ on $\R^d\times \R^d$ is called a coupling operator of $\mathscr{L}^x$ and itself if
\begin{align*}
  \widetilde{\mathscr{L}} h(y,y')&=\mathscr{L}^x h(y)\quad \text{if $h\in C_b^2(\R^d)$ and independent of $y'$},\\
  \widetilde{\mathscr{L}} h(y,y')&=\mathscr{L}^x h(y')\quad \text{if $h\in C_b^2(\R^d)$ and independent of $y$}.
\end{align*}
The coefficients of any coupling operator must be of the form $\widetilde{\mathscr{L}}\sim (a_x(y,y'),\tilde f_x(y,y'))$ with
\begin{equation*}
  a_x(y,y')=\begin{pmatrix}
    G(x,y) &c_x(y,y')\\ c_x(y,y')& G(x,y')
  \end{pmatrix},\quad \tilde f_x(y,y')=\begin{pmatrix}
    f(x,y)\\ f(x,y')
  \end{pmatrix},
\end{equation*}
where $c_x(y,y')$ is a matrix such that $a_x(y,y')$ is nonnegative definite. When taking $c_x(y,y')=0$, $\widetilde{\mathscr{L}}$ is called an independent coupling of $\mathscr{L}^x$ and itself, denoted by $\widetilde{\mathscr{L}}_{ind}$. The monotonicity condition \eqref{a-4} means that
\begin{equation}\label{d-1}
\widetilde{\mathscr{L}}_{ind}(|y-y'|^2) =2(y\!-\!y')\cdot(f(x,y)\!-\!f(x,y'))\!+\!\|g(x,y) \!-\!g(x,y')\|^2\!\leq -\beta |y-y'|^2.
\end{equation}
Together with the following coercivity condition, i.e.
\[\mathscr{L}^x (|y|^2)=-y\cdot f(x,y)+\|g(x,y)\|^2\leq c_1 -\beta_1 |y|^2\]
for some $c_1,\beta_1>0$,
similar to Lemma \ref{lem-3.2} below, we obtain from the monotonicity condition \eqref{d-1} that
\[\W_1(P_t^x(y,\cdot),\pi^x)\leq \W_2(P_t^x(y,\cdot),\pi^x)\leq \frac{c_1}{\beta_1} \e^{-\beta t},\quad t>0,\]
and hence (E2) holds.

\begin{mylem}\label{lem-3.2}
Assume $\mathrm{(B2)}$ holds.
Let $\rho:[0,\infty)\to [0,\infty)$ be in $C^2([0,\infty))$ satisfying $\rho(0)=0$, $\rho'>0$, $\rho''<0$, and $\rho(x)\to \infty$ as $x\to \infty$. There is a constant $c_2>0$ such that $\rho(x)>c_2 x$ for $x\geq 0$. If there exist constants $c_3,\,\beta_2,\,\beta_3>0$, a coupling $\widetilde{\mathscr{L}}^x$ of $\mathscr{L}^x$ and itself such that
\begin{equation}\label{d-2} \widetilde{\mathscr{L}}^x\rho(|y-y'|)\leq -\beta_2\rho(|y-y'|),\quad x,y,y'\in\R^d,
\end{equation}
and
\begin{equation}\label{d-3}
\mathscr{L}^x \rho(|y|)\leq c_3-\beta_3\rho(|y|), \quad x,y\in \R^d.
\end{equation}
Then for any $t>0$
\[\W_1(P_t^x(y,\cdot),\pi^x)\leq  \frac{c_3}{\beta_3 c_2} \e^{-\beta_2 t},\quad x,y\in\R^d.   \]
\end{mylem}

\begin{proof}
  Associated with the coupling operator $\widetilde{\mathscr{L}}^x$, there is a coupling process $(Y_t^x,\tilde Y_t^x)$ with initial value $Y_0^x=y$ and $\tilde Y_0^x=\xi$, where $\xi$ is a random variable with distribution $\pi^x$.
  By virtue of \eqref{d-3}, it holds
  \begin{align*}
  \E[\rho(|Y_t^{x,y}|)]
  &\leq \E[\rho(|Y_0^{x,y}|)]+\int_0^t\!\big(c_3-\beta_3\E[\rho (|Y_s^{x,y}|)]\big)\d s.
  \end{align*}
  Gronwall's inequality yields that
  \begin{equation}\label{d-4}
  \E[\rho(|Y_t^{x,y}|)]\leq \frac{c_3}{\beta_3}\big(1-\e^{-\beta_3t}\big)+\rho(|y|) \e^{-\beta_3 t}.
  \end{equation}
  This implies the existence of invariant probability measure $\pi^x$ for the process $(Y_t^{x,y})$. Letting $t\to \infty$ in \eqref{d-4}, Fatou's lemma implies that
  \begin{equation}\label{d-5}
  \int_{\R^d} \rho(|y'|)\pi^x(\d y')\leq \frac{c_3}{\beta_3}.
  \end{equation}
  Note that the uniform nondegerate condition (B2) yields that $\pi^x$ admits a density w.r.t. the Lebesgue measure.

  By It\^o's formula and \eqref{d-2}, for $0\leq s\leq t$,
  \begin{align*}
    \E\big[\rho(|Y_t^x-\tilde Y_t^x|)\big]
    &\leq \E\big[\rho(|Y_s^x-\tilde Y_s^x|)\big]-\beta_2\int_s^t\! \E\big[\rho(|Y_r^x-\tilde Y_r^x|)\big]\d r.
  \end{align*}
  Using Gronwall's inequality,
  we get
  \[\E\big[\rho(|Y_t^x-\tilde Y_t^x|)\big]\leq \E\big[\rho(|y-\xi|)\big]\e^{-\beta_2 t}.\]
  Hence, by \eqref{d-5},
  \begin{align*}
  \W_1(P_t^x(y,\cdot),\pi^x)&\leq c_2^{-1}\int_{\R^d}\! \rho(|y-y'|)\pi^x(\d y') \e^{-\beta_2 t}\\
  &=
  c_2^{-1} \int_{\R^d}\!\rho(|y'|)\pi^x(\d y') \e^{-\beta_2 t}
   \leq \frac{c_3}{\beta_3 c_2} \e^{-\beta_2 t}.
  \end{align*}
  The proof is completed.
\end{proof}

According to Stroock and Varadhan \cite{SV79} (see, also, \cite[Chapter IV]{IW}), when $\bar{b}$, $\bar{\sigma}$ are continuous satisfying the linear growth condition, and  $\bar{\sigma}$ is nondegenerate,  SDE
\begin{equation}\label{limit}
\d \bar{X}_t=\bar{b}(\bar{X}_t)\d t+\bar{\sigma}(\bar{X}_t)\d B_t,\quad \bar{X}_0=x_0,
\end{equation}
admits a weak solution whose distribution is unique. Of course, the wellposedness of SDE \eqref{limit} is  a precondition of the averaging principle. Corollaries \ref{cor-1} and \ref{cor-2} provide us different sufficient conditions for the wellposedness of SDE \eqref{limit}.

Denote
\begin{equation}\label{gen-lim}
\begin{aligned}
\bar{\mathscr{L}} h(x)&=\bar{b}(x)\cdot \nabla h(x)+\frac 12\mathrm{tr}\big(\bar{a}(x)\nabla^2h(x)\big)\\
&=\sum_{k=1}^d\bar{b}_k(x)\frac{\partial h(x)}{\partial x_k}+\frac 12\sum_{k,l=1}^d\bar{a}_{kl}(x)\frac{\partial^2 h(x)}{\partial x_k\partial x_l},\quad h\in C^2(\R^d),
\end{aligned}
\end{equation} where $\bar{a}(x)=(\bar{\sigma}\bar{\sigma}^\ast)(x)$.

Let us introduce some notations. For $T>0$, $\mathcal{C}([0,T];\R^d)$ denotes the set of continuous functions from $[0,T]$ to $\R^d$ endowed with uniform norm, i.e. $\|x_\cdot\|_\infty=\sup_{t\in [0,T]}|x_t|$ for $x_\cdot\in \mathcal{C}([0,T];\R^d)$.  Denote by $\mathcal{L}_{X^\ve}$ and $\mathcal{L}_{\bar{X}}$ the law of the stochastic processes $(X_t^\ve)$ and $(\bar{X}_t)$ in $\mathcal{C}([0,T];\R^d)$ respectively. Let
\begin{equation}\label{b-3}
\mathscr{L} h(x,y)=b(x,y)\cdot\nabla h(x)+\frac 12 \mathrm{tr}\big((\sigma\sigma^\ast)(x,y)\nabla^2h(x)\big),\quad h\in C^2(\R^d), x,y\in \R^d.
\end{equation}

\begin{mythm}\label{thm-1}
Let $(X_t^\ve,Y_t^\ve)$ be the solution to \eqref{a-1} and $(\bar{X}_t)$ the solution to \eqref{limit}. Assume $\mathrm{(E1),\, (A1)}$-$\mathrm{(A3)}$, and $\mathrm{(B1)}$ hold.
Then for any $T>0$ the process $(X_t^\ve)_{t\in [0,T]}$ converges weakly in $\mathcal{C}([0,T];\R^d)$ as $\veps\to 0$ to the process $(\bar{X}_t)_{t\in [0,T]}$.
\end{mythm}

\begin{proof}
By virtue of Corollary \ref{cor-2}, $\bar b$ and $\bar a$ are continuous in $x$. By \cite[Lemma 5.2.1]{SV79}, $\bar\sigma$ as the square root of $\bar a$ is given by an absolutely convergent series in terms of $\bar a$ and hence $\bar \sigma$ is also continuous in $x$. Then, the limit process $(\bar{X}_t)$ exists and is unique in distribution.
  Due to the linear growth condition (A2), it is standard to show
  \[\E\big[\sup_{t\in [0,T]} |X_t^\ve|^p\big]\leq C(T, x_0,p),\quad \text{for $p\geq 1$}.\]
  By It\^o's formula,
  \begin{align*}
    \E|X_t^\ve-X_s^\ve|^4&\leq 8\E\Big|\int_s^t\! b(X_r^\ve,Y_r^\ve)\d r\Big|^4+8\E\Big|\int_s^t\! \sigma(X_r^\ve,Y_r^\ve)\d W_r\Big|^4\\
    &\leq 8(t-s)^3\E\int_s^t\!|b(X_r^\ve,Y_r^\ve)|^4\d r+288 (t-s)\E\int_s^t\!|\sigma(X_r^\ve, Y_r^\ve)|^4\d r\\
    &\leq C(t-s)^2
  \end{align*}for some constant $C>0$. Combining this with the fact $X_0^\ve=x_0$, the collection of the laws of $\{(X_t^\ve)_{t\in [0,T]};\ve>0\}$ over $\mathcal{C}([0,T];\R^d)$ is tight by
  \cite[Theorem 12.3]{Bill}. As a consequence, there is a subsequence $\{\mathcal{L}_{X^{\ve'}};\ve'>0\}$ and a limit law $\mathcal{L}_{\wt X}$ over $\mathcal{C}([0,T];\R^d)$ such that $\mathcal{L}_{X^{\ve'}}$ converges weakly to $\mathcal{L}_{\wt X}$ as $\ve'\to 0$. According to Skorokhod's representation theorem, we may assume that $(X_t^{\ve'})_{t\in [0,T]}$ converges almost surely to some random process $(\wt X_t)_{t\in [0,T]}$ in $\mathcal{C}([0,T];\R^d)$ as $\ve'\to 0$.

  We proceed to characterize the process $(\wt X_t)_{t\in [0,T]}$. To this end, we shall prove that for any $h\in C_c^2(\R^d)$, the space of functions with compact support and continuous second order derivatives,
  \begin{equation}\label{b-4}
  h(\wt X_t)-h(x_0)-\int_0^t\bar{\mathscr{L}} h(\wt X_s)\d s\ \ \text{is a martingale},
  \end{equation} where $\bar{\mathscr{L}}$ is defined in \eqref{gen-lim}.
  This means that $(\wt{X}_t)$ is a solution to SDE \eqref{limit}, and hence $\law_{\wt X}$ equals to $\law_{\bar{X}}$ due to the uniqueness of distribution for solutions to SDE \eqref{limit}.

  To prove \eqref{b-4}, it is suffices to show that for $0\leq s<t\leq T$, for any bounded $\F_s$ measurable function $\Phi$,
  \begin{equation}\label{b-5}
  \E\Big[\Big(h(\wt X_t)-h(\wt X_s)-\int_s^t\!\bar{\mathscr{L}}h(\wt X_r)\d r\Big)\Phi\Big]=0,\quad \forall h\in C_c^2(\R^d).
  \end{equation}
  As a solution to SDE \eqref{a-1}, $(X_t^{\ve'})$ satisfies that
  \[\E\Big[\Big(h(X_t^{\ve'})-h(X_s^{\ve'})-\int_s^t\mathscr{L}^{\ve'} h(X_r^{\ve'}, Y_r^{\ve'})\d r\Big)\Phi\Big]=0.\]
  The almost sure convergence of $(X_t^{\ve'})_{t\in [0,T]}$ to $(\wt X_t)_{t\in [0,T]}$ in $\mathcal{C}([0,T];\R^d)$ yields that
  \[\lim_{\ve'\to 0}\E\Big[ \big(h(X_t^{\ve'})-h(X_s^{\ve'})\big)\Phi\Big]=\E\Big[\big(h(\wt X_t)-h(\wt X_s)\big)\Phi\Big].
  \]
  Hence, for \eqref{b-5} we only need to show
  \begin{equation*}
    \lim_{\ve'\to 0}\E\Big[\int_s^t\!\big(\mathscr{L}^{\ve'} h(X_r^{\ve'}, Y_r^{\ve'}) -\bar{ \mathscr{L}} h(\wt X_r)\big)\d r\Big|\F_s\Big]=0.
  \end{equation*}
  In view of the expression of $\mathscr{L}^{\ve'}$ and $\bar{\mathscr{L}}$, we need to show
  \begin{gather}\label{b-6}
    \lim_{\ve'\to 0}\E\Big[\int_s^t\! \big(b(X_r^{\ve'}, Y_r^{\ve'})\cdot\nabla h(X_r^{\ve'})-\bar{b}(\wt X_r)\cdot\nabla h(\wt X_r)\big)\d r\Big|\F_s\Big]=0,\\ \label{b-7}
    \lim_{\ve'\to 0}\E\Big[\int_s^t\!\Big(\mathrm{tr}\big((\sigma\sigma^\ast)(X_r^{\ve'}, Y_r^{\ve'} )\nabla^2h(X_r^{\ve'})\big)-\mathrm{tr}\big(\bar{a}(\wt X_r)\nabla^2h(\wt X_r)\big)\Big)\d r\Big|\F_s\Big]=0.
  \end{gather}
  We shall use the time discretization method and the coupling method to show \eqref{b-6} and \eqref{b-7}. Since the method is similar, we only present the proof of \eqref{b-6}.

  For $\delta \in (0,1)$, let $r(\delta)=s+\big[\frac{r-s}{\delta}\big]\delta$ for $r\in [s,t]$, where $\big[\frac{r-s}{\delta}\big]$ denotes the integer part of $\frac{r-s}{\delta}$.
  \begin{align*}
    &\!\E\Big[\int_s^t\!\!\big(b(X_r^{\ve'},Y_r^{\ve'})\cdot\nabla h(X_r^{\ve'})-\bar{b}(\wt X_r)\cdot h(\wt X_r)\big)\d r\Big|\F_s\Big]\\
    &=\E\Big[\!\int_s^t\!\!\Big\{b(X_r^{\ve'},Y_r^{\ve'}) \cdot(\nabla h(X_r^{\ve'})\!-\!\nabla h(\wt X_r))\!+\!\big(b(X_r^{\ve'},Y_r^{\ve'})\!- \!b(X_{r(\delta)}^{\ve'},Y_r^{\ve'})\big)\cdot \nabla h(\wt X_r)\\
    &\qquad\quad\! + \big( b(X_{r(\delta)}^{\ve'},Y_r^{\ve'}) - \bar{b}(X_{r(\delta)}^{\ve'})\big)\cdot \nabla h(\wt X_{r(\delta)})\\
    &\qquad\quad\!  +\! \big( b(X_{r(\delta)}^{\ve'},Y_r^{\ve'}) \!-\!\bar{b}(\wt X_{r(\delta)} ) \big) \cdot   \big(\nabla h(\wt X_r ) - \nabla h(\wt X_{r(\delta)} )\big) \\ &\qquad \quad + \big(\bar{b}(X_{r(\delta)}^{\ve'}) -  \bar{b}(\wt X_r)\big)\! \cdot\!  \nabla h(\wt X_r)\Big\}\d r\Big|\F_s\Big]
  \end{align*}
  By the continuity  of $b$ and $\bar{b}$ due to (A1) and Corollary  \ref{cor-1}, it follows from the almost sure convergence of $(X_r^{\ve'})$ to $(\wt X_r)$ in $\mathcal{C}([0,T];\R^d)$ that
  \begin{align*}
  \lim_{\ve',\delta\to 0}\!\E&\Big[\!\int_s^t\!\Big\{b(X_r^{\ve'},Y_r^{\ve'})\cdot(\nabla h(X_r^{\ve'})-\nabla h(\wt X_r))+\big(b(X_r^{\ve'},Y_r^{\ve'})-b(X_{r(\delta)}^{\ve'},Y_r^{\ve'})\big)\cdot \nabla h(\wt X_r)\\
  &\qquad +\big( b(X_{r(\delta)}^{\ve'},Y_r^{\ve'}) \!-\!\bar{b}(X_{r(\delta)}^{\ve'})  \big)\!\cdot \! \big(\nabla h(\wt X_r )\!-\!\nabla h(\wt X_{r(\delta)} )\big) \\
  & \quad \quad +\!\big(\bar{b}(X_{r(\delta)}^{\ve'}) - \bar{b}(\wt X_r)\big)\cdot \nabla h(\wt X_r)\Big\}\d r\Big|\F_s\Big]\\
  &=0.
  \end{align*}
  Therefore, to prove \eqref{b-6} we only to show that for suitable choice of $\delta$,
  \begin{equation}\label{b-7.5}
  \E\Big[\int_s^t\!(b(X_{r(\delta)}^{\ve'},Y_r^{\ve'}) -\bar{b}(X_{r(\delta)}^{\ve'}))\cdot
  \nabla h(\wt{X}_{r(\delta)})\d r\Big|\F_s\Big]\longrightarrow 0,\quad \text{as $\ve'\to 0$.}
  \end{equation}
  Let $N_t=[(t-s)/\delta]$, $s_k+s+k\delta$ for $0\leq k\leq N_t$ and $s_{N_t+1}=t$. Then,
  \begin{equation}\label{b-8}
  \begin{aligned}
    &\E\Big[\int_s^t\!(b(X_{r(\delta)}^{\ve'},Y_r^{\ve'}) -\bar{b}(X_{r(\delta)}^{\ve'}))\cdot
  \nabla h(\wt{X}_{r(\delta)})\d r\Big|\F_s\Big] \\
   &=\sum_{k=0}^{N_t} \E\Big[\E\Big[\int_{s_k}^{s_{k+1}}\!\!(b(X_{s_k}^{\ve'}, Y_r^{\ve'})-\bar{b}(X_{s_k}^{\ve'}))\cdot \nabla h(\wt X_{s_k})\d r\Big|\F_{s_k}\Big]\Big|\F_s\Big].
  \end{aligned}
  \end{equation}
  On each time interval $[s_k,s_{k+1})$, $0\leq k\leq N_t$, we introduce a coupling process $(Y_r^{\ve'},\wt Y_r^{\ve',k})$ by the following SDEs
  \begin{equation}\label{b-9}
  \begin{cases}
    \d Y_r^{\ve'}=\frac 1{\ve'}f(X_r^{\ve'}, Y_r^{\ve'}) \d r+\frac1{\sqrt{\ve'}} g(Y_r^{\ve'})\d B_r,\\
    \d \wt Y_r^{\ve',k}=\frac 1{\ve'} f(X_{s_k}^{\ve'}, \wt{Y}_r^{\ve',k})\d r+\frac 1{\sqrt{\ve'}} g(\wt{Y}_r^{\ve',k})\d B_r, \quad \wt{Y}_{s_k}^{\ve',k}=Y_{s_k}^{\ve'}.
  \end{cases}
  \end{equation}
  Notice that under the conditional expectation $\E[\,\cdot\,|\F_{s_k}]$, the distribution of $\wt{Y}_r^{\ve',k}$ equals to $P_{(r-s_k)/\ve'}^{X_{s_k}^{\ve'}}(Y_{s_k}^{\ve'},\cdot)$ due to the homogeneity and the uniqueness of solution to \eqref{b-9}.
  It follows from  (A2), (B1) and It\^o's formula that for $r\in [s_k,s_{k+1})$,
  \begin{align*}
    \E\big[|Y_r^{\ve'}-\wt{Y}_r^{\ve',k}|^2\big|\F_{s_k}\big]&\leq \frac C{\ve'}\int_{s_k}^r\E\big[|X_u^{\ve'}-X_{s_k}^{\ve'}|^2+|Y_u^{\ve'} -\wt{Y}_u^{\ve',k}|^2\big|\F_{s_k}\big]\d u\\
    &\leq \frac{C}{\ve'}\int_{s_k}^r (u-s_k)+\E\big[|Y_u^{\ve'} -\wt{Y}_u^{\ve',k}|^2\big|\F_{s_k}\big] \d u \\
    &\leq \frac{C\delta^2}{\ve'}+ \frac{C}{\ve'}\int_{s_k}^r\!\E\big[|Y_u^{\ve'} -\wt{Y}_u^{\ve',k}|^2\big|\F_{s_k}\big] \d u.
  \end{align*}
  By Gronwall's inequality,
  \begin{equation}\label{b-10}
  \E\big[|Y_r^{\ve'}-\wt{Y}_r^{\ve',k}|^2\big|\F_{s_k}\big]\leq \frac{C\delta^2}{\ve'}\e^{C(r-s_k)/\ve'}.
  \end{equation}
  Therefore, using (A1), (A2), (E1), and \eqref{b-10},
  \begin{equation}\label{b-11}
  \begin{aligned}
    &\Big|\E\Big[\int_{s_k}^{s_{k+1}}\!\!(b(X_{s_k}^{\ve'}, Y_r^{\ve'})-\bar{b}(X_{s_k}^{\ve'}))\cdot \nabla h(\wt{X}_{s_k})\d r\Big|\F_{s_k}\Big]\Big|\\
    &  \leq |\nabla h|_\infty \E\Big[\int_{s_k}^{s_{k+1}}\!\!|b(X_{s_k} ^{\ve'}, Y_r^{\ve'})-b(X_{s_k}^{\ve'},\wt{Y}_{r}^{\ve',k})|\d r\Big|\F_{s_k}\Big]\\
    &\quad  +\!\Big|\E\Big[\int_{s_k}^{s_{k+1}}\!\! \big(b(X_{s_k}^{\ve'},\wt{Y}^{\ve',k}_r)- \bar{b}(X_{s_k}^{\ve'}) \big)\cdot \nabla h(\wt X_{s_k})\d r\Big|\F_{s_k}\Big]\Big|\\
    &\leq |\nabla h|_\infty\!\!\int_{s_k}^{s_{k+1}}\!\!\! \E\big[K_1|Y_r^{\ve'}\!\!-\! \wt{Y}_r^{\ve'\!,k}|^2\big|\F_{s_k} \big]^{\frac 12}\d r \\ &\qquad  + K_2|\nabla h|_\infty \int_{s_k}^{s_{k+1}}\!   \|P_{\frac{r-s_k}{\ve'}}^{X_{s_k}^{\ve'}}( Y_{s_k}^{\ve'},\cdot)\!-\!\pi^{X_{s_k}^{\ve'}}\|_\var \d r \\
    &\leq |\nabla h|_\infty\!\int_{s_k}^{s_{k+1}}\!\!\! \big\{ \frac{C\delta \sqrt{K_1}}{\sqrt{\ve'}} \e^{\frac{C(r-s_k)}{2\ve'}} +K_2 C_0\e^{-\kappa\frac{r-s_k}{\ve'}} \big\} \d r\\
    &\leq |\nabla h|_{\infty}\!\frac{C\delta^2}{\sqrt{\ve'}}\e^{ \frac{C\delta}{2\ve'}} +|\nabla h|_{\infty}\!K_2 C_0 \int_{0}^{\delta}\e^{-\kappa\frac{r}{\ve'}} \d r.
  \end{aligned}
  \end{equation} Inserting this estimate into \eqref{b-8}, we obtain that
  \begin{equation}\label{b-12}
  \begin{aligned}
    &\Big|\E\Big[\int_s^t\!(b(X_{r(\delta)}^{\ve'},Y_r^{\ve'})-\bar{b}(X_{r(\delta)}^{\ve'}))\cdot
  \nabla h(\wt{X}_r)\d r\Big|\F_s\Big]\Big|\\
  &\leq |\nabla h|_\infty\!\sum_{k=0}^{N_t}\!\Big(\frac{C\delta^2}{ \sqrt{\ve'}}\e^{\frac{C\delta}{2\ve'}}+C_0K_2   \int_0^{\delta}\e^{-\kappa\frac{r}{\ve'}}\d r\Big)\\
  &\leq C|\nabla h|_\infty \frac{t\delta}{\sqrt{\ve'}} \e^{\frac{C\delta}{2\ve'}}+C|\nabla h|_\infty \frac{t\ve'}{\kappa\delta}\big( 1-\e^{-\kappa\frac{\delta}{\ve'}}
  \big).
  \end{aligned}
  \end{equation}
  Take $\delta=\ve'\ln\ln\big(\frac1{\ve'}\big)$, then
  \[\lim_{\ve'\to 0}\frac{\delta}{\ve'}=\infty,\quad \lim_{\ve'\to 0}\frac{\delta}{\sqrt{\ve'}}\e^{\frac{C\delta}{2\ve'}}=\lim_{\ve'\to 0} \sqrt{\ve'} \big(\ln\ln\big(\frac1{\ve'}\big)\big)\Big(\ln\big(\frac 1{\ve'}\big)\Big)^{\frac{C}{2}}=0.\]
  Hence, using this choice of $\delta$, we get from \eqref{b-12} that
  \[\lim_{\ve'\to 0} \E\Big[\int_s^t\!(b(X_{r(\delta)}^{\ve'},Y_r^{\ve'})-\bar{b}(X_{r(\delta)}^{\ve'}))\cdot
  \nabla h(\wt{X}_{r(\delta)})\d r\Big|\F_s\Big]=0.\]
  This is the desired \eqref{b-7.5}, and  further \eqref{b-6} holds.

  Consequently, we have shown that $(X_{t}^{\ve'})_{t\in [0,T]}$ converges weakly to $(\bar{X}_t)_{t\in [0,T]}$. The arbitrariness of weakly convergent subsequence of $(X_t^{\ve'})_{t\in [0,T]}$ and the uniqueness of $(\bar{X}_t)_{t\in [0,T]}$ imply that the whole sequence $(X_t^{\ve})_{t\in [0,T]}$ converges weakly to $(\bar{X}_t)_{t\in [0,T]}$. We have completed the proof.
\end{proof}
%
%

\begin{mythm}\label{thm-2}
Assume $\mathrm{(E2),\, (B1)}$, $\mathrm{(A1)}$-$\mathrm{(A3)}$ hold.
Then for any $T>0$ $(X_t^\ve)_{t\in [0,T]}$ converges weakly in $\mathcal{C}([0,T];\R^d)$ to $(\bar{X}_t)_{t\in [0,T]}$ as $\ve\to 0$.
\end{mythm}

\begin{proof}
  This theorem can be proved along the line of the argument of Theorem \ref{thm-1}. We only point out the difference caused by using the $L^1$-Wasserstein distance instead of the total variation distance in ergodicity condition.
  Namely, instead of \eqref{b-11}, we now have
  \begin{equation}\label{b-13}
  \begin{aligned}
    &\Big|\E\Big[\int_{s_k}^{s_{k+1}}\!\!(b(X_{s_k}^{\ve'}, Y_r^{\ve'})-\bar{b}(X_{s_k}^{\ve'}))\cdot \nabla h(\wt X_{s_k})\d r\Big|\F_{s_k}\Big]\Big|\\
    &\leq |\nabla h|_\infty \E\Big[\int_{s_k}^{s_{k+1}}\!\!|b(X_{s_k} ^{\ve'}, Y_r^{\ve'})\!-\!b(X_{s_k}^{\ve'},\wt{Y}_{r}^{\ve',k})|\!+\! |b(X_{s_k}^{\ve'},\wt{Y}_{r}^{\ve'\!,k})\!-\! \bar{b}(X_{s_k}^{\ve'})|\d r\Big|\F_{s_k}\Big]\\
    &\leq |\nabla h|_\infty\!\!\int_{s_k}^{s_{k+1}}\!\!\! \big\{ \E\big[K_1|Y_r^{\ve'}\!\!-\!\wt{Y}_r^{\ve'\!,k}|^2\big|\F_{s_k} \big]^{\frac 12}\!+\! \sqrt{K_1}\W_1\big(P_{\frac{r-s_k}{\ve'}}^{X_{s_k}^{\ve'}}( Y_{s_k}^{\ve'},\cdot), \pi^{X_{s_k}^{\ve'}}\big)  \big\}\d r\\
    &\leq |\nabla h|_\infty\!\int_{s_k}^{s_{k+1}}\!\!\! \big\{ \frac{C\delta \sqrt{K_1}}{\sqrt{\ve'}} \e^{\frac{C(r-s_k)}{2\ve'}} +\sqrt{K_1} \kappa_2 \e^{-\lambda_2\frac{r-s_k}{\ve'}}   \big\} \d r\\
    &\leq |\nabla h|_{\infty}\!\frac{C\delta^2}{\sqrt{\ve'}}\e^{ \frac{C\delta}{2\ve'}} +|\nabla h|_{\infty} \kappa_2\sqrt{K_1} \int_{0}^{\delta}\e^{-\lambda_2\frac{r}{\ve'}} \d r.
  \end{aligned}
  \end{equation}
  Other details are omitted to save space. Then this theorem can be proved.
\end{proof}


We have proposed appropriate conditions to ensure the fully  coupled system $(X_t^\veps,Y_t^\veps)$ satisfies the averaging principle in the weak convergence sense. In this study of slow-fast systems, there are also some works to study the averaging principle in the $L^2$-convergence sense  with the restriction that the diffusion coefficient of the slow process does not depending on the fast process, such as \cite{Liu}. Inspired by an example in \cite{ZYB},  we shall illustrate that generally the averaging principle in the $L^2$-convergence sense does not hold for the slow-fast systems with the diffusion coefficients of the slow processes depending on the fast processes. In \cite{ZYB}, the fast process is a Markov chain on the state space $\S=\{1,2\}$, and hence its transition probability and invariant probability measure can be explicitly calculated. Now we can generalize this result to any ergodic fast processes under the help of the ergodic theorem.

Let $(Y_t)_{t\geq 0}$ be a stochastic process on a general state space $\S$ which could be $\N$ or $\R^d$. Assume $(Y_t)_{t\geq 0}$ is ergodic with the invariant probability measure $\mu$, and the ergodic theorem holds, i.e.
\[\lim_{T\to \infty}\frac 1T\int_0^T h(Y_t)\d t=\mu(h), \quad \forall \,h\in \B_b(\S).\]
Let $Y_t^\veps=Y_{t/\veps}$ be the fast process with parameter $\veps$. Consider the slow process $(X_t^\veps)_{t\geq 0} $:
\begin{equation}\label{sf-1}
\d X_t^\veps=\sigma(Y_t^\veps)\d W_t,\quad X_0^\veps=x_0\in\R^d,
\end{equation}
where $(W_t)$ is a $d$-dimensional Brownian motion, $\sigma:\S\to \R^d\times \R^d$ is bounded measurable. $\sigma$ is not a constant function.

The averaging principle in weak convergence sense tells us that $(X_t^\veps)_{t\geq0}$ shall converge  weakly to the limit process
\begin{equation}\label{sf-2}
\d \bar{X}_t=\bar{\sigma}\d W_t,\quad \bar{X}_0=x_0,
\end{equation}
where $\bar{\sigma}$ is a constant matrix satisfying $\bar{\sigma}\bar{\sigma}^\ast=\int_{\S}(\sigma\sigma^\ast)(y)\mu(\d y)$. However,
\begin{align*}
  \E|X_t^\veps-\bar{X}_t|^2&=\E\Big[\int_0^t\!\|\sigma(Y_s^\veps)-\bar{ \sigma}\|^2\d s\Big]=t\E\Big[\frac1{t/\veps}\int_0^{t/\veps} \|\sigma(Y_s)-\bar{\sigma} \|^2\d s\Big].
\end{align*}
Applying the ergodic theorem to the ergodic process $(Y_t)_{t\geq 0}$ and the bounded function $h(y):=\|\sigma(y)-\bar{\sigma}\|^2$, we obtain that
\begin{equation}\label{sf-3}
\lim_{\veps\to 0}\E|X_t^\veps-\bar{X}_t|^2=t\int_{\S}\|\sigma(y)-\bar{\sigma}\|^2\mu(\d y).
\end{equation}
No matter what constant matrix $\bar{\sigma}$ equals to, the right hand side of the previous equation cannot be $0$, and hence $X_t^\veps$ cannot converge in $L^2$ to any process $\bar{X}_t$ in the form \eqref{sf-2}.
Consequently, we see from the simple process \eqref{sf-1} that the averaging principle in $L^2$-convergence sense will not hold for a fully coupled two time-scale stochastic system.

\vskip 2em
\noindent\textbf{Conflict of Interest} The authors declare  no conflict of interest.

\end{document}